\documentclass[11pt,noinfoline]{article}

\RequirePackage[OT1]{fontenc}
\RequirePackage[
amsthm,amsmath
]{imsart}

\usepackage{graphicx}
\usepackage{latexsym,amsmath}
\usepackage{amsmath,amsthm,amscd}
\usepackage{amsfonts}
\usepackage[psamsfonts]{amssymb}

\usepackage{mathabx} 

\usepackage{enumerate}

\usepackage{url}

\usepackage{tocvsec2}

\usepackage{epstopdf}
\usepackage{wrapfig}
\usepackage{float}
\usepackage[small]{caption}
\usepackage{hyperref}

\startlocaldefs
\numberwithin{equation}{section}

\theoremstyle{plain} 
\newtheorem{theorem}{Theorem}[section]

\newtheorem{proposition}[theorem]{Proposition}

\theoremstyle{definition} 

\theoremstyle{definition} 

\newtheorem*{ex*}{Example}
\theoremstyle{remark} 

\theoremstyle{remark} 
\newtheorem{remark}[theorem]{Remark}
\newtheorem*{remark*}{Remark}

\numberwithin{equation}{section}

\makeatletter
\def\subsubsubsection{\@startsection{subsubsubsection}{4}{\z@}{-3.25ex plus -1ex minus -.2ex}{1.5ex plus .2ex}{\normalsize}}
\makeatother                                   
 
\newcommand{\beqa}{\begin{eqnarray}}
\newcommand{\eeqa}{\end{eqnarray}}
 
\newcommand{\bseq}{\begin{subequations}}
 
\newcommand{\eseq}{\end{subequations}}

\newcommand{\dd}{\partial}

\renewcommand{\dd}{{\,\operatorname{d}}}

\newcommand{\pv}{\operatorname{p{.}v{.}}}

\newcommand{\al}{\alpha}

\newcommand{\vka}{\varkappa}
\newcommand{\la}{\lambda}

\newcommand{\be}{\beta}
\newcommand{\De}{\Delta}

\renewcommand{\Psi}{\overline{\Phi}}

\newcommand{\ffrown}{\text{\raisebox{3pt}[0pt][0pt]{$\frown$}}}

\renewcommand{\O}{\underset{\ffrown}{<}}

\renewcommand{\u}{\mathsf{u}}
\renewcommand{\nu}{\mathsf{nu}}

\newcommand{\ii}{\operatorname{I}}

\renewcommand{\P}{\operatorname{\mathsf{P}}} 
 
\newcommand{\E}{\operatorname{\mathsf{E}}}

\newcommand{\R}{\mathbb{R}}

\newcommand{\C}{\mathbb{C}}
\newcommand{\CC}{\mathbb{C}}

\newcommand{\GG}{\mathfrak{G}}

\newcommand{\vp}{\varepsilon}

\renewcommand{\le}{\leqslant}
\renewcommand{\ge}{\geqslant}

\renewcommand{\cdot}{\#}

\newcommand{\fn}{f_{1n}}
\newcommand{\fnj}[1]{f_{1n}^{(#1)}}
\newcommand{\gn}{g_{1n}}
\newcommand{\gnj}[1]{g_{1n}^{(#1)}}
\newcommand{\hn}{h_{1n}}
\newcommand{\hnj}[1]{h_{1n}^{[#1]}}

\newcommand{\dnj}[1]{d_{1n}^{(#1)}}

\renewcommand{\Re}{\operatorname{\mathrm{Re}}}
\renewcommand{\Im}{\operatorname{\mathrm{Im}}}

\endlocaldefs

\begin{document}

\begin{frontmatter}

\title{More on the nonuniform Berry--Esseen bound}
\runtitle{Nonuniform Berry--Esseen}

%

\begin{aug}
\author{\fnms{Iosif} \snm{Pinelis}\thanksref{t2}\ead[label=e1]{ipinelis@mtu.edu}}
  \thankstext{t2}{Supported by NSA grant H98230-12-1-0237}
\runauthor{Iosif Pinelis}


\address{Department of Mathematical Sciences\\
Michigan Technological University\\
Houghton, Michigan 49931, USA\\
E-mail: \printead[ipinelis@mtu.edu]{e1}}
\end{aug}

\begin{abstract}
Certain smoothing inequalities were proposed in the recent paper posted on arXiv at \url{http://arxiv.org/abs/1301.2828} in order to lessen the 
large gap between the best correctly established upper and lower bounds on the constant factor in the nonuniform Berry--Esseen bound. 
As an illustration of the possible uses of those inequalities, a quick proof of Nagaev's classical nonuniform bound was given there. 
Here we describe another, apparently more effective class of smoothing inequalities, and give a yet quicker and simpler proof of Nagaev's result. 
\end{abstract}

  
%

\setattribute{keyword}{AMS}{AMS 2010 subject classifications:}




\end{frontmatter}

\settocdepth{chapter}

\tableofcontents 

\settocdepth{subsubsection}

\theoremstyle{plain} 
\numberwithin{equation}{section}


\section{Uniform and nonuniform Berry--Esseen (BE) bounds}\label{intro} 

Suppose that $X_1,\dots,X_n$ are independent zero-mean random variables (r.v.'s), with 
\begin{equation*}
	S:=X_1+\dots+X_n,\ A:=\sum\E|X_i|^3<\infty,\ \text{and}\  B:=\sqrt{\sum\E|X_i|^2}>0. 
\end{equation*}
Consider  
\begin{equation*}
	\De(z):=\textstyle{|\P(S>Bz)-\P(Z>z)|}\quad\text{and}\quad r_L:=A/B^3, 
\end{equation*}
where $Z\sim N(0,1)$ and $z\ge0$; of course, $r_L$ is the so-called Lyapunov ratio. 
Note that, in the ``iid'' case (when the $X_i$'s are iid), $r_L$ will be on the order of $1/\sqrt n$. 

In such an iid case, let us also assume that 
$\E X_1^2=1$. 

Uniform and nonuniform BE bounds are upper bounds on $\De(z)$ of the forms  
\begin{equation*}
	c_\u\,r_L\quad\text{and}\quad c_\nu\,\frac{r_L}{1+z^3}, 
\end{equation*}
respectively, 
for some absolute positive real constants $c_\u$ and $c_\nu$ and for all $z\ge0$. 

Apparently the best currently known upper bound on $c_\u$ (in the iid case) is due to Shevtsova \cite{shev11} and is given by the inequality 
$
	c_\u\le0.4748. 
$ 
On the other hand, Esseen \cite{esseen56} 
showed that $c_\u$ cannot be less than $\frac{3+\sqrt{10}}{6 \sqrt{2 \pi }}=0.4097\ldots$. 
\label{c_u gap} 

Thus, the optimal value of $c_\u$ is known to be within the rather small interval from $0.4097$ to $0.4748$ in the iid case \big(in the general case the best known upper bound on $c_\u$ appears to be $0.5600$, due to 
Shevtsova \cite{shevtsova-DAN};
a slightly worse upper bound, $0.5606$, is due to Tyurin \cite{tyurin}\big). 
So, the gap factor $0.4748/0.4097\ldots$ or even $0.5600/0.4097\ldots$ between the best known upper and lower bounds on (the least possible value of) $c_\u$ is now rather close to 1. 

As discussed in \cite{nonunif}, the situation is quite different for the absolute constant factor $c_\nu$ in the corresponding nonuniform BE bound. 
Namely, the best correctly established upper bound on $c_\nu$ in the iid case is 
over 25 times the corresponding best known lower bound, and this gap factor is greater than 
31 in the general case.  
Certain approaches were proposed in \cite{nonunif} to lessen this gap. 
Here we describe another, apparently more effective approach toward the same goal. 

\section{The Bohman--Prawitz--Vaaler smoothing inequalities}\label{praw}
Let us say that a function $F$ is a scaled distribution function (scaled d.f., for brevity) if $F=\la F_0$ for some real $\la\ge0$ and some d.f.\ $F_0$.  
To a significant extent the mentioned best known uniform BE bounds are based 
on the smoothing result due to Prawitz \cite[(1a, 1b)]{prawitz72_limits}, which can be stated as  follows. There exists a nonempty class of functions $M\colon\R\to\CC$ such that 
\begin{equation}\label{eq:M=0}
	M(t)=0\quad\text{if}\quad|t|>1 
\end{equation}
and 
for any scaled d.f.\ $F$, any real $T>0$, and any real $x$, 
\begin{gather}
	\tfrac12\,F(\infty-)+\GG\big(M_T(-\cdot)f_F(\cdot)\big)(x)
	\le  F(x-)\le F(x+)
	\le\tfrac12\,F(\infty-)+\GG\big(M_T(\cdot)f_F(\cdot)\big)(x), 
	\label{eq:praw}\\ 
\intertext{where $f_F$ denotes the Fourier-Stieltjes transform $\int_\R e^{ix\cdot}\dd F(x)$ of $F$,}\quad 
	M_T(\cdot):=M(\cdot/T), \label{eq:M_T} 
	\\ 
\GG(f)(x):=\frac i{2\pi}\,\pv\int_\infty^\infty e^{-itx}f(t)\frac{\dd t}t,  \label{eq:GG}
\end{gather}
and $\pv$ stands for ``principal value'', so that $\pv\int_{-\infty}^\infty:=\lim_{\vp\downarrow0\atop A\uparrow\infty}\big(\int_{-A}^{-\vp}+\int_\vp^A\big)$; here and subsequently, the symbol $\cdot$ stands for the argument of a function. 
In fact, this is a trivial restatement of Prawitz's result, which latter was presented for the case when $F$ itself is a d.f. 
Of course, the upper and lower bounds in \eqref{eq:praw} must take on only real values; this can be provided by the condition that 
\begin{equation}\label{eq:M1,M2}
	M_1:=\Re M\ \text{is even}\quad\text{and}\quad M_2:=\Im M\ \text{is odd.}
\end{equation}  
Note also that the upper and lower bounds in \eqref{eq:praw} easily follow from each other, by changing $X$ to $-X$. 
Functions $M$ satisfying conditions \eqref{eq:praw} and \eqref{eq:M=0} may be referred to as \emph{bounding smoothing filters}; accordingly, the corresponding inverse Fourier transforms $\check M(\cdot)=\frac1{2\pi}\int_\R e^{-it\cdot}M(t)\dd t$ may be referred to as \emph{bounding smoothing kernels}.  

One particular bounding smoothing filter $M$ 
was given by Prawitz \cite{prawitz72_limits} and can be defined by the formula 
\begin{equation}\label{eq:M special}
	M(t)=\big[(1-|t|)\,\pi t\cot\pi t+|t|
	-i(1-|t|)\,\pi t\big]\,\ii\{|t|<1\} 
\end{equation}
for all $t\ne0$; here and subsequently, it is tacitly assumed that the functions of interest are extended to $0$ by continuity. 

\begin{remark}\label{rem:M praw}
The derivative $M'$ of the Prawitz bounding smoothing filter $M$ as in \eqref{eq:M special} is a function of bounded variation. 
So, twice integrating $\int_\R e^{-it\cdot}M(t)\dd t$ by parts and using the Riemann--Lebesgue lemma, one can easily see that the corresponding bounding smoothing kernel $\check M$ 
is such that $x^2\check M(x)-\sin x\to0$ as $|x|\to\infty$ and hence $\int_\R|\check M(x)|\dd x<\infty$. 
Thus, Prawitz's particular $M$ is the Fourier transform of a function $\check M\in L^1(\R)$. 
%
\end{remark}

Earlier, inequalities of the form \eqref{eq:praw} were obtained by Bohman~\cite{bohman} for another class of functions $M$. 
Another approach to Prawitz's results was demonstrated by Vaaler~\cite{vaaler85}. 
For more on ways to construct functions $M$ satisfying the conditions \eqref{eq:praw}, \eqref{eq:M=0}, and \eqref{eq:M1,M2} and at that of any prescribed degree of smoothness, see \cite{nonunif}. 

It was also shown in \cite{nonunif} how one can use \eqref{eq:praw} to obtain smoothing inequalities, which work better in the tail zones, provided that the function $M$ is smooth enough. 
Here we shall present an alternative construction of such smoothing inequalities, taking into account possible large deviations. With this new approach, the two main differences (and hence potential advantages) are that we will (i) have fewer terms to bound and (ii) be able to make use of not so smooth bounding smoothing filters $M$; for instance, Prawitz's $M$ as in \eqref{eq:M special}, while optimal in a certain sense, is not smooth enough to be used in \cite{nonunif} --- but it can certainly be used within the framework of the method described in the next section.


\section{Another construction of smoothing inequalities for nonuniform BE bounds}\label{new constr}

Take any natural $k$ and any r.v.\ $X$ such that $\E|X|^k<\infty$, and introduce the functions $L_X$, $F_X$, $G_X$, $\hat F_X$, $\hat G_X$ defined by the formulas  
\begin{gather*}
	L_X(x):=x^k\big(\P(X>x)\ii\{x>0\}-\P(X<x)\ii\{x<0\}\big), \\ 
\begin{alignedat}{2}	
	F_X(x):=&\E X^k \ii\{X\le x\},&\quad G_X(x):=&\E(x_+\wedge X)^k, \\
	\hat F_X(t):=&\E X^k e^{itX}, &\quad 
	\hat G_X(t):=&
	\int_0^1k\al^{k-1}\E X^k e^{it\al X}\dd\al
\end{alignedat}
\end{gather*}
for all real $x$ and $t$; of course, the definition of the function $L_X$ makes sense even without the condition $\E|X|^k<\infty$. 
Here and subsequently, we employ the usual notation 
\begin{equation}
	x_+:=0\vee x\quad\text{and}\quad x_-:=0\wedge x=-(-x)_+.
\end{equation}


Note that 
\begin{equation}\label{eq:parity}
L_{-X}(\cdot)=(-1)^{k+1} L_X(-\cdot),\quad
	\hat F_{-X}(\cdot)=(-1)^k \hat F_X(-\cdot),\quad\text{and}\quad 
	\hat G_{-X}(\cdot)=(-1)^k \hat G_X(-\cdot).  
\end{equation}

For a moment, consider the particular case when the r.v.\ $X$ is nonnegative. Then 
\begin{enumerate}[(i)]
	\item $L_X=G_X-F_X$;  
	\item $F_X$ and $G_X$ are scaled d.f.'s, with $F_X(\infty-)=\E X^k=G_X(\infty-)$; also, $G_X$ is continuous on $\R$;
	\item $\hat F_X$ and $\hat G_X$ are the Fourier-Stieltjes transforms of $F_X$ and $G_X$, respectively; 
\end{enumerate}
To check item (ii) on this list, use the dominated convergence. To verify item (iii) concerning $\hat G_X$ and $G_X$, note that $G_X(x)=\mu\big((-\infty,x]\big)$ for all real $x$ and $\int_\R e^{itx}\mu(\dd x)=\hat G_X(t)$ for all real $t$, where $\mu$ is the nonnegative measure defined by the condition 
$\int_\R h\dd\mu=\E\int_\R k z^{k-1}\ii\{0<z\le X\}h(z)\dd z$ for all bounded and/or nonnegative Borel functions $h\colon\R\to\R$; the relation between $\hat F_X$ and $F_X$ is only easier to check. 


Removing now the temporary assumption that the r.v.\ $X$ is nonnegative and recalling \eqref{eq:praw}, one has  
\begin{align}
& \GG\big(M_T(-\cdot)\hat G_{X_+}(\cdot)\big)(x)
	-\GG\big(M_T(\cdot)\hat F_{X_+}(\cdot)\big)(x) \notag \\ 
	\le & L_{X_+}(x+)	\le  L_{X_+}(x-) \label{eq:<L} \\ 
	\le & \GG\big(M_T(\cdot)\hat G_{X_+}(\cdot)\big)(x)
	-\GG\big(M_T(-\cdot)\hat F_{X_+}(\cdot)\big)(x)  \label{eq:L<} 
\end{align}
for all real $x$. 
Using these inequalities (with $-X$ and $-x$ in place of $X$ and $x$) together with the parity  properties \eqref{eq:parity} and $\GG\big(f(-\cdot)\big)(-x)=-\GG(f)(x)$, one obtains the ``negative'' counterpart of the upper bound in \eqref{eq:L<}:  
\begin{align}
	L_{X_-}(x+)&=(-1)^{k+1}L_{(-X)_+}\big((-x)-\big) 
	\notag \\ 
	&\le(-1)^{k+1}\Big[\GG\Big(M_T\big((-1)^{k+1}\cdot\big)\hat G_{(-X)_+}(\cdot)\Big)(-x)
	-\GG\Big(M_T\big((-1)^k\cdot\big)\hat F_{(-X)_+}(\cdot)\Big)(-x)\Big] 
	\notag \\ 
	&=\GG\Big(M_T\big((-1)^k\cdot\big)\hat G_{X_-}(\cdot)\Big)(x)
	-\GG\Big(M_T\big(-(-1)^k\cdot\big)\hat F_{X_-}(\cdot)\Big)(x). \label{eq:L<,-} 
\end{align}

To proceed further, suppose that for 
some real constant $\vka$
\begin{equation}\label{eq:M-1}
\text{the function $\big(M(\cdot)-\vka\big)/\cdot$ is in $L^1([-1,1])$}; 	
\end{equation}
this condition was assumed in \cite{prawitz72_limits} and will be satisfied in the applications, usually with $\vka=1$; it is even unclear whether the conditions \eqref{eq:praw}, \eqref{eq:M=0}, and \eqref{eq:M1,M2} can ever all hold without \eqref{eq:M-1}. 
Introducing the functions 
\begin{equation*}
	M_{j,T}(\cdot):=M_j(\cdot/T)
\end{equation*}
for $j\in\{1,2\}$ (cf.\ \eqref{eq:M_T} and \eqref{eq:M1,M2}) and using \eqref{eq:M-1} and the fact that $\big|\int_\vp^A\frac{\sin zu}u\dd u\big|$ is bounded uniformly over all real $z$ and all $\vp$ and $A$ such that $0<\vp<A$, one can easily show (cf.\ \cite[(5)]{prawitz72_limits}) that the limit $\GG(M_{j,T}f)(x)$ exists (and is) in $\C$ for any $j\in\{1,2\}$, any real $T>0$, any characteristic function (c.f.) $f$, and any real $x$. 
This allows one to recombine the terms in the upper bounds in \eqref{eq:L<} and \eqref{eq:L<,-} to see that for any real $x\ge0$ 
\begin{align}
	x^k\P(X\ge x)=L_{X_+}(x-)=&L_{X_+}(x-)+L_{X_-}(x+) \notag \\ 
	\le& \GG\big(M_T(\cdot)\hat G_{X_+}(\cdot)\big)(x)
	-\GG\big(M_T(-\cdot)\hat F_{X_+}(\cdot)\big)(x) \notag \\
	&+\GG\Big(M_T\big((-1)^k\cdot\big)\hat G_{X_-}(\cdot)\Big)(x)
	-\GG\Big(M_T\big(-(-1)^k\cdot\big)\hat F_{X_-}(\cdot)\Big)(x) \notag \\ 
	=&\GG\big(M_{1,T}\,[\hat G_{X_+}+\hat G_{X_-}-\hat F_{X_+}-\hat F_{X_-}]\big)(x) \notag \\ 
	&+i\,\GG\big(M_{2,T}\,\big[\hat G_{X_+}+(-1)^k\hat G_{X_-}+\hat F_{X_+}+(-1)^k\hat F_{X_-}\big]\big)(x) \notag \\ 
	=&\GG\big(M_{1,T}\E X^k(W_X-V_X)\big)(x)+i\,\GG\big(M_{2,T}\E|X|^k(W_X+V_X)\big)(x), 
		\label{eq:le GG(g,f)}
\end{align}
where 
\begin{equation}\label{eq:f_a,g_a}
	V_X(\cdot):=e^{iX\cdot}\quad\text{and}\quad
	W_X(\cdot):=\int_0^1V_X(\al\cdot)k\al^{k-1}\dd\al=\int_0^1e^{i\al X\cdot}k\al^{k-1}\dd\al;   
\end{equation}
here we also used the obvious identities $\hat F_{X_\pm}(\cdot)=\E X_\pm^k e^{iX\cdot}$, $\hat G_{X_\pm}(\cdot)=\int_0^1k\al^{k-1}\E X_\pm^k e^{i\al X\cdot}\dd\al$, $X_+^k+X_-^k
=X^k$, and $X_+^k+(-X_-)^k
=|X|^k$. 

Quite similarly to the upper bound on $x^k\P(X\ge x\}$ in \eqref{eq:le GG(g,f)}, one can derive the corresponding lower bound on $x^k\P(X>x\}$, with $-M_{2,T}$ in place of $M_{2,T}$. Thus, one obtains 

\begin{theorem}\label{th:}
Let $M$ be any function such that the condition \eqref{eq:praw} holds for all d.f.'s $F$, all real $x$, and all real $T>0$, as well as the conditions \eqref{eq:M=0}, \eqref{eq:M1,M2}, 
and \eqref{eq:M-1}. 
Then for all real $x\ge0$ and all r.v.'s $X$ such that $\E|X|^k<\infty$
\begin{equation}\label{eq:2sided}
	\big|x^k\P(X\ge x)
	-\GG\big(M_{1,T}\,\E X^k (W_X-V_X)\big)(x)\big|
	\le i\,\GG\big(M_{2,T}\,\E|X|^k (W_X+V_X)\big)(x),  
\end{equation}
where $W_X$ and $V_X$ are as in \eqref{eq:f_a,g_a}; 
inequality \eqref{eq:2sided} also holds with $\P(X>x)$ in place of $\P(X\ge x)$. 
\end{theorem}


\begin{remark}\label{rem:|X|} 
Introducing the c.f.\ of $X$,
\begin{equation*}
	f(\cdot):=\E e^{i X\cdot},   
\end{equation*}
and its $k$th derivative $f^{(k)}$, 
one has
$\E X^k V_X=i^{-k}f^{(k)}$ and $(\E X^k W_X)(\cdot)=i^{-k}\int_0^1 f^{(k)}(\al\cdot)k\al^{k-1}\dd\al$, whence, concerning the left-hand side of \eqref{eq:2sided},   
\begin{equation}\label{eq:exprs} 
	\E X^k(W_X-V_X)(\cdot)=i^{-k}\int_0^1 [f^{(k)}(\al\cdot)-f^{(k)}(\cdot)]k\al^{k-1}\dd\al. 
\end{equation}
Somewhat unfortunately, when $k$ is odd the expression of the function $\E|X|^k (W_X+V_X)$ in the right-hand side of \eqref{eq:2sided} in terms of the c.f.\ $f$ (cf.\ e.g.\ \cite{positive}) is much less convenient than the expression for $\E X^k(W_X-V_X)$ in \eqref{eq:exprs} -- and the case most interesting in the applications is that of $k=3$. 
\end{remark}

However, there is a simple and apparently rather effective 
way to deal with this inconvenience: 


\begin{proposition}\label{prop:fix} 
Under the conditions of Theorem~\ref{th:}, 
\begin{multline}
\Big|\GG\big(M_{2,T}\,\E|X|^k (W_X+V_X)\big)(x)
-i^{-k}\int_0^1 k\al^{k-1}\,\GG\big(M_{2,T}(\cdot)\,[f^{(k)}(\al\cdot)+f^{(k)}(\cdot)]\big)(x)\dd\al\Big| \\
\le
\frac{c_{2,p}}\pi\,\frac{2k-p}{k-p}\,\E\frac{|X_-|^k}{(|X_-|+x)^p}\,\frac1{T^p}
\le
\frac{c_{2,p}}\pi\,\frac{2k-p}{k-p}\,
\Big(\E |X_-|^{k-p}\bigwedge\frac{\E |X_-|^k}{x^p}\Big)\frac1{T^p} \label{eq:fix}
\end{multline}
for all $p\in(0,k)$ and all real $x>0$, where 
\begin{equation*}
	c_{2,p}:=\sup_{u\in\R}|u|^p|\widehat{N_2}(u)|  
\end{equation*}
and $\widehat{N_2}$ is the Fourier transform of the function $N_2(\cdot):=M_2(\cdot)/\cdot$. 
\end{proposition}

\newcommand{\lhs}{\text{LHS}}

\begin{proof}[Proof of Proposition~\ref{prop:fix}] 
Note that the left-hand side of the first inequality in \eqref{eq:fix} equals  
$\lhs:=\big|\GG\big(M_{2,T}\,\E|X|^k (W_X+V_X)\big)(x)
-\GG\big(M_{2,T}\,\E X^k (W_X+V_X)\big)(x)\big|
=2\big|\GG\big(M_{2,T}\,\E X_-^k (W_X+V_X)\big)(x)\big|$; cf.\ \eqref{eq:exprs}. 
Further, by \eqref{eq:GG}, 
\begin{align*}
%
\lhs&=\frac1{\pi}\,\Big|\int_\R e^{-itx}\frac{M_2(t/T)}t\,\E X_-^k 
\Big(\int_0^1k\al^{k-1}e^{i\al tX}\dd\al+e^{itX}\Big)\dd t\Big| \\ 
&=\frac1{\pi}\,\Big|\E X_-^k\,\int_\R N_2(u) 
\Big(\int_0^1k\al^{k-1}e^{i\al Tu(X-x)}\dd\al+e^{iTu(X-x)}\Big)\dd u\Big| \\ 
&=\frac1{\pi}\,\Big|\E |X_-|^k\, 
\Big(\int_0^1k\al^{k-1}\widehat{N_2}\big(\al T(X-x)\big)\dd\al+\widehat{N_2}\big(T(X-x)\big)\Big)\dd u\Big| \\ 
&\le\frac{c_{2,p}}{\pi}\,\E\frac{|X_-|^k}{T^p|X-x|^p}\, 
\Big(\int_0^1k\al^{k-1-p}\dd\al+1\Big)\\ 
&=\frac{c_{2,p}}{\pi}\,\E\frac{|X_-|^k}{(|X_-|+x)^p}\, 
\frac{2k-p}{k-p}\,\frac1{T^p}; 
\end{align*}
here we used the equality $|X-x|=|X_-|+x$, valid   
for any real 
$x>0$ on the event $\{X_-\ne0\}$. 
Thus, the first inequality in \eqref{eq:fix} is verified, and the second inequality there follows because $|X_-|+x\ge |X_-|\vee x$ for $x>0$. 
\end{proof}

\begin{remark}\label{rem:M2 praw}
For instance, for Prawitz's particular function $M$ as in \eqref{eq:M special}, 
$N_2(\cdot)=-\pi(1-|\cdot|)_+$, so that $\widehat{N_2}(x)=-\pi\big(\frac{\sin x/2}{x/2}\big)^2$ for real $x\ne0$, whence $c_{2,2}=4\pi$ and $c_{2,1}=4\pi\sup_{x>0}\frac{\sin^2 x/2}x$. 
It follows that for $k=3$ the second upper bound in \eqref{eq:fix} is 
$16\Big(\E |X_-|\bigwedge\frac{\E |X_-|^3}{x^2}\Big)\frac1{T^2}$ if $p$ is taken to be $2$, and it is no greater than 
$3.6231\Big(\E |X_-|^2\bigwedge\frac{\E |X_-|^3}x\Big)\frac1T$ with $p=1$. 
Here one can obviously further bound the moments of $|X_-|$ from above by the corresponding moments of $|X|$; these bounds can be obviously improved if the distribution of $X$ is symmetric. 
%
For simplicity, let us consider here the iid case and accordingly let $X:=S/\sqrt n$, so that $X$ is a zero-mean unit-variance r.v.  Then, with $k=3$ and $p=2$,  
\begin{equation}\label{eq:E rat1}
\E\frac{|X_-|^3}{(|X_-|+x)^2}
\le \E |X_-|\bigwedge\frac{\E |X_-|^3}{x^2}
\le 1\bigwedge\frac{\E |X|^3}{x^2}
\le1\bigwedge\frac{2+\be_3/\sqrt n}{x^2}  	
\end{equation}
by the Rosenthal-type inequality (see e.g.\ \cite[Lemma 6.3]{chen-shao05} or \cite[(12)]{pin12-2smooth}) 
\begin{equation}\label{eq:rosenthal}
	\E|X|^3\le2+\be_3/\sqrt n,   
\end{equation}
where $\be_3:=\E|X_1|^3$; this may be compared with $\E|Z|^3=2\sqrt{\frac2\pi}\approx1.6$, where $Z\sim N(0,1)$. 
In view of Markov's inequality and the mentioned value $0.4748$ of $c_\u$, \eqref{eq:rosenthal} also yields  
%
the 
classical result by Nagaev \cite{nagaev65}
\begin{equation}\label{eq:BE nonunif}
	|\P(S>x\sqrt n)-\P(Z>x)|\le c_\nu\frac{\be_3}{(1+x^3)\sqrt n}
\end{equation}
for all real $x\ge0$ with $c_\nu=4.5$ in the ``small $n$'' case when $\frac{\be_3}{\sqrt n}\ge\frac23$. So, if one recalls \cite{nonunif} that the apparently best known upper bound on $c_\nu$ in the iid case is over 
$25$ and thus considers 
the value $4.5$ for $c_\nu$ satisfactory at this point, then without loss of generality (w.l.o.g.) $\frac{\be_3}{\sqrt n}<\frac23$. 
Moreover, comparing the desired nonuniform bound $4.5\frac{\be_3}{(1+x^3)\sqrt n}$ with the known uniform bound $0.4748\frac{\be_3}{\sqrt n}$, one sees that w.l.o.g.\ $x>x_0:=(\frac{4.5}{0.4748}-1)^{1/3}=2.039\dots$. 

Another upper bound on the term $\E\frac{|X_-|^k}{(|X_-|+x)^p}$, which is apparently better than the upper bound in \eqref{eq:E rat1}, is as follows. Again, let us consider the case of principal interest, when $k=3$. At that, to be specific, take $p=2$. 
Consider the function $h(\cdot):=h_x(\cdot):=\frac{|\cdot_-|^3}{(|\cdot_-|+x)^2}$, for any given real $x>0$. Then it is easy to see that $|h'''(u)|\le\frac6{x^2}$ 
for all real nonzero $u$. 
Hence, by Tyurin's result \cite[Theorem~2]{tyurinSPL}, 
\begin{equation}\label{eq:E rat2}
\E\frac{|X_-|^3}{(|X_-|+x)^2}\le\frac{\psi(x)+\be_3/\sqrt n}{x^2}, 	
\end{equation}
where $\psi(x):=x^2\E\frac{|Z_-|^3}{(|Z_-|+x)^2}$, 
so that the function $\psi$ is increasing on the interval $(0,\infty)$, from $\psi(0+)=0$ to $\psi(\infty-)=\E|Z_-|^3=\sqrt{\frac2\pi}=0.797\ldots<0.8$. 
%
Thus, the upper bound in \eqref{eq:E rat2} is less than $\frac{0.8+2/3}{x_0^2}<0.36<1$ and hence indeed significantly less than the upper bound in \eqref{eq:E rat1}. 
Note also that the increase of $\psi$ is rather slow; in particular, $\psi(3.5)\approx0.35$, whereas certain considerations show that the most ``difficult'' values of $x$ are between $x_0\approx2$ and about $3.5$. 

One can also try to use the more accurate upper bound 
\begin{equation}
\frac1{\al^pT^p||X_-|+x|^p}\,
	\sup\{u^p|\widehat{N_2}(u)|\colon u\ge\al T x\}
\end{equation}
on $\widehat{N_2}\big(\al T(|X_-|+x)\big)$ -- instead of the bound $\frac1{\al^pT^p||X_-|+x|^p}\,\sup\{|u|^p|\widehat{N_2}(u)|\colon u\in\R\}$, essentially used in the proof of Proposition~\ref{prop:fix}.  
%
At that, one may want to utilize a function $M$ with its imaginary part $M_2$ smoother than that of the Prawitz particular function, so that the Fourier transform $\widehat{N_2}$ of the function $N_2(\cdot):=M_2(\cdot)/\cdot$ be decreasing faster. 
\end{remark}  
 

In \cite{nonunif}, a quick proof of \eqref{eq:BE nonunif} was given. Using Theorem~\ref{th:} in the present paper (with $k=3$), we can now give the following, yet quicker proof of \eqref{eq:BE nonunif}, in which we have fewer terms to bound than in the ``quick proof'' in \cite{nonunif}.

\paragraph{}
\label{quick-proof}
\emph{A quicker proof of Nagaev's nonuniform BE bound \eqref{eq:BE nonunif}}. 
Let $T=c_T\sqrt{n}/\be_3$, where 
$c_T$ is a small enough positive real constant. Let $A\O B$ mean $|A|\le CB$ for some absolute constant $C$. 
Let $X:=S/\sqrt n$. 
If $T\le1$ then $1\O\frac{\be_3}{\sqrt n}$. So, for all real $x\ge0$, by the Markov and Rosenthal inequalities, 
$(1+x^3)\P(X\ge x)\le1+\E|X|^3\O 1+\frac{\be_3}{\sqrt n}\O\frac{\be_3}{\sqrt n}$ and similarly 
$(1+x^3)\P(Z\ge x)\O\frac{\be_3}{\sqrt n}$, whence \eqref{eq:BE nonunif} follows. 

It remains to consider the case $T>1$. 
Note that then $n>(\be_3/c_T)^2\ge3$ and hence $n\ge4$ provided that $c_T\le1/\sqrt3$. 

In view of the uniform BE bound, Theorem~\ref{th:} (with $M$ as in \eqref{eq:M special}, say), \eqref{eq:exprs}, Proposition~\ref{prop:fix}, and Remark~\ref{rem:M2 praw}, in order to prove \eqref{eq:BE nonunif} 
it is enough to show that 
$\GG_{1\al}(f'''-g''')\O\frac{\be_3}{\sqrt n}$ and $\GG_{2\al}(f''')\O\frac{\be_3}{\sqrt n}$ for $\al\in(0,1]$, 
where 
$f$ is the c.f.\ of $X:=S/\sqrt n$, $g(\cdot):=e^{-\cdot^2/2}$, and  
\begin{equation*}
	\GG_{j\al}(h)(x):=
\GG\big(M_j(\tfrac{\al\cdot}T)h(\al\cdot)\big)(x).
\end{equation*}
 
For $j\in\{0,1,2,3\}$, introduce $f_1^{(j)}(t):=
\big(\frac{\dd}{\dd t}\big)^j f_1(t)$ and $\fnj j(t):=f_1^{(j)}(t/\sqrt n)$, where $f_1$ denotes the c.f.\ of $X_1$. 
Similarly, starting with $g_1:=g$ in place of $f_1$, define $\gnj j$, and then let 
$\dnj j:=\fnj j-\gnj j$ and $\hnj j:=\big|\fnj j\big|\vee\big|\gnj j\big|$; omit superscripts ${}^{(0)}$ and ${}^{[0]}$. 
Note that $f=\fn^n$ and hence $\sqrt{n}f'''=f_{31}+f_{32}+f_{33}$, where 
$f_{31}:=(n-1)(n-2)\fn^{n-3}\big(\fnj1\big)^3$, $f_{32}:=3(n-1)\fn^{n-2}\fnj1\fnj2$, and $f_{33}:=\fn^{n-1}\fnj3$; do similarly with $g$ and $g_1$ in place of $f$ and $f_1$. 
By Remark~\ref{rem:M praw} and \cite[Proposition~4.3]{nonunif}, $M_j(\tfrac{\al\cdot}T)f_{33}/\be_3$ is a quasi-c.f.\ and hence, by \cite[Proposition~4.2]{nonunif}, $\GG_{j\al}(f_{33})\O\be_3$, for $j\in\{1,2\}$. 

So, it suffices to show that $\GG_{1\al}(f_{3k}-g_{3k})\O\be_3$ and $
\GG_{2\al}(f_{3k})\O\be_3$ for $k\in\{1,2
\}$.
%
This can be done in a straightforward manner using the following estimates for $j\in\{0,1,2,3\}$ and $|t|\le T$:\quad 
$M_1\O1$, $
M_2(\frac t{
T})\O
\frac{|t|}{
T}\O|t|\be_3/\sqrt n$, $\hn(t)^{n-j}\le e^{-ct^2}$ (where $c$ is a positive real number depending only on the choice of $c_T$), 
$\hnj1(t)\O|t|
/\sqrt n$, $\hnj2(t)\O1$, $|\dnj j(t)|\O\be_3(|t|/\sqrt n)^{3-j}$, and hence 
$\fn^{n-j}(t)-\gn^{n-j}(t)\O|t|^3 e^{-ct^2}\be_3/\sqrt n$; cf.\ e.g.\ \cite[Ch.\ V, Lemma~1]{pet75}. 
For instance, $|f_{31}-g_{31}|\O n^2(D_{311}+D_{312})$, 
where 
$D_{311}(t):=
\big(|\fn^{n-3}-\gn^{n-3}|\big(\hnj1\big)^3\big)(t)\O|t|^3 e^{-ct^2}\frac{\be_3}{\sqrt n}\big(\frac{|t|}{\sqrt n}\big)^3$ and 
$D_{312}(t):=\big(\hn^{n-3}\big(\hnj1\big)^2|\dnj1|\big)(t) 
\O e^{-ct^2}\,\big(\frac{|t|}{\sqrt n}\big)^2\,\be_3\big(\frac{|t|}{\sqrt n}\big)^2$, so that  
$\GG_{1\al}(f_{31}-g_{31})
\O 
\int_{-\infty}^\infty(t^6+t^4)e^{-ct^2}\be_3\,\frac{\dd t}{|t|}
\O\be_3$.  
\qed

\bibliographystyle{abbrv}


\bibliography{C:/Users/Iosif/Dropbox/mtu/bib_files/citations12.13.12}
\end{document}